\documentclass[letterpaper,12pt]{amsart}
\title{A Gr\"obner Basis for the Secant Ideal of the Second Hypersimplex}

\author{ Seth Sullivant}
\address{Department of Mathematics \\ North Carolina State University, Raleigh, NC 27695}
\email{smsulli2@ncsu.edu}

\date{}

\usepackage{latexsym,array,delarray,amsthm,amssymb,epsfig}

\theoremstyle{plain}
\newtheorem{thm}{Theorem}[section]
\newtheorem{lemma}[thm]{Lemma}
\newtheorem{prop}[thm]{Proposition}
\newtheorem{cor}[thm]{Corollary}

\theoremstyle{definition}
\newtheorem{defn}[thm]{Definition}
\newtheorem{ex}[thm]{Example}

\theoremstyle{remark}

\newcommand{\zz}{\mathbb{Z}}
\newcommand{\nn}{\mathbb{N}}
\newcommand{\pp}{\mathbb{P}}

\newcommand{\cc}{\mathbb{C}}
\newcommand{\kk}{\mathbb{K}}

\newcommand{\bfa}{\mathbf{a}}

\newcommand{\bfi}{\mathbf{i}}
\newcommand{\bfj}{\mathbf{j}}

\newcommand{\bfu}{\mathbf{u}}
\newcommand{\bfv}{\mathbf{v}}

\newcommand{\bfx}{\mathbf{x}}

\newcommand{\conv}{\mathrm{conv}}
\newcommand{\ind}{\mbox{$\perp \kern-5.5pt \perp$}}

\begin{document}
\maketitle

\begin{abstract}
We determine a Gr\"obner basis for the secant ideal of the toric ideal associated to the second hypersimplex $\Delta(2,n)$, with respect to any circular term order.  The Gr\"obner basis of the secant ideal requires polynomials of odd degree up to $n$.  This shows that the circular term order is $2$-delightful, resolving a conjecture of Drton, Sturmfels, and the author.
The proof uses Gr\"obner degenerations for secant ideals, combinatorial characterizations of the secant ideals of monomial ideals, and the relations between secant ideals and prolongations.
\end{abstract}


\section{Introduction}

If $X \subset  \pp^{m-1}$ is a projective variety, its $r$th secant variety $X^{\{r\}} \subseteq \pp^{m-1}$ is the closure of the  union of all planes in $\pp^{m-1}$ spanned by $r$ points in $X$.  There is a large literature on secant varieties, and the vast majority of results focus on computing their dimension \cite{Alexander1995, Catalisano2002}.  Inspired by problems in computational complexity and algebraic statistics more attention has been paid to the problem of determining the vanishing ideals $I(X^{\{r\}})$ of secant varieties \cite{Landsberg2004, Simis2000, Sturmfels2006}.  

This paper presents a case study of the secant ideals $I(X^{\{2\}})$ of a particular family of toric varieties associated to the second hypersimplices 
$$\Delta(2,n) =  \conv( \{e_i + e_j  \, \, | \, \, 1 \leq i < j \leq n \} ).$$
The associated toric variety $X_{2,n}$ arises in algebraic geometry as the closure of the torus orbit of a generic point on the Grassmannian $Gr_{2,n}$.  The secant varieties $X_{2,n}^{\{r\}}$ arise in statistics as the projectivization of the Zariski closure of the parameter space of the factor analysis model, with $r$-factors \cite{Drton2007}.

Our main result is the computation of a Gr\"obner basis for the secant ideal $I(X_{2,n}^{\{2\}})$, with respect to a certain circular term order, confirming a conjecture of Drton, Sturmfels, and the author \cite{Drton2007}.  The proof relies on the ``delightful'' strategy described by Sturmfels and the author \cite{Sturmfels2006} plus the connection betweens secant ideals and prolongations, introduced in the work of Landsberg and Manivel \cite{Landsberg2003} and extended in the work of Sidman and the author \cite{Sidman2006}.  As a corollary to these arguments, we also deduce a Gr\"obner basis for the symbolic square of the second hypersimplex.

\section*{Acknowledgments}
Seth Sullivant was partially supported by NSF grant DMS-0840795.


\section{Initial Ideal of the Second Hypersimplex}

In this section, we introduce the circular term order $\prec$ and describe the Gr\"obner basis  and initial ideal of the second hypersimplex, recalling results from \cite{Deloera1995}.  The quadratic squarefree initial ideal ${\rm in}_\prec(I_n)$ has a simple combinatorial description in terms of non-crossing edges in the circular straight-line drawing of the complete graph $K_n$.  Then recalling results on secant ideals and symbolic powers of edge ideals from \cite{Simis2000, Sturmfels2006, Sullivant2007}, we give combinatorial descriptions of the ideals ${\rm in}_\prec(I_n)^{\{2\}}$ and ${\rm in}_\prec(I_n)^{(2)}$.

Throughout the remainder of the paper, we use the notation $I^{\{r\}}$ to denote the \emph{secant ideal} of the ideal $I$.  If $I$ is a radical ideal in a polynomial ring over an algebraically closed field, then $I^{\{r\}}$ is the vanishing ideal of the $r$-th secant variety of $V(I)$.  The notion of secant ideal extends beyond both radical ideals and algebraic fields, and the definitions can be found in \cite{Simis2000, Sturmfels2006}, though we will not need these here.

Let $\cc[x] :=  \cc[x_{ij} \, \, | \, \, 1 \leq i < j \leq n ]$ and $\cc[t] :=  \cc[t_i \, \, | \, \, 1 \leq i \leq n]$ and let $\phi_n$ be the ring homomorphism: 
$$\phi_n :  \cc[x] \rightarrow \cc[t],  \quad  x_{ij} \mapsto t_i t_j.$$
The toric ideal $I_n  = \ker \phi_n$ is the vanishing ideal of the toric variety of the second hypersimplex.  

We will often need to work with the combinatorial structure of a certain circular embedding of the complete graph $K_n$.  We consider the vertices of $K_n$ as the $n$-th roots of unity in the complex plane.  Each edge $(i,j)$ connects two of the roots of unity.   This drawing of $K_n$ is as the set of all diagonals (including edges) of a regular convex $n$-gon in the plane.  

The edges of $K_n$ fall into $\lfloor \frac{n}{2}  \rfloor$ orbits under the action of the dihedral group $D_n$ on the $n$th roots of unity.  Let the $i$th class consist of the edges that are equivalent to the edge $1i$, for $i \in \{2, \ldots, \lceil (n+1)/2 \rceil \}$.  This also divides the variables $x_{ij}$ into $\lfloor \frac{n}{2} \rfloor$ classes.  

\begin{figure}[h]
\includegraphics[width = 10cm]{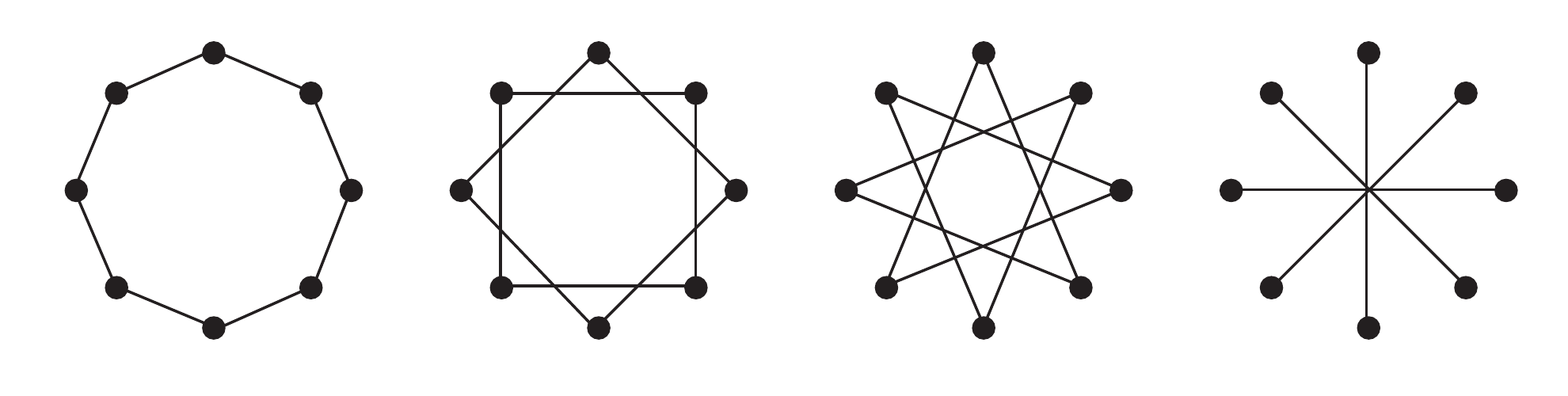} \caption{The four orbits of edges in the circular embedding of $K_8$}
\end{figure}

\begin{defn}
A circular term order $\prec$ is any block term order such that $x_{i_1 j_1} \succ x_{i_2j_2}$ whenever the edge $i_1 j_1$ is in a smaller class than $i_2 j_2$.  
\end{defn}

In other words, large variables in the block ordering correspond to edges that are close to the boundary of the polygon, and small edges cut deep through its interior.  In figure 1, the four orbits of edges in $K_8$ are arranged in decreasing weight in the circular term order.  De Loera, Sturmfels, and Thomas \cite{Deloera1995} characterized the Gr\"obner basis for the second hypersimplex with respect to any circular term order.

\begin{thm}
The set of quadratic binomials
$$\left\{  \underline{x_{ij}x_{kl}} - x_{ik}x_{jl}, \underline{x_{il}x_{jk}} - x_{ik}x_{jl} \, \, | \, \, 1 \leq i < j < k <l \leq  n  \right\}$$
form a reduced Gr\"obner basis for $I_n$ with respect to any circular term order. 
\end{thm}

We say that a pair of edges $ij$ $kl$ \emph{cross} if the line segments in the circular drawing intersect (including at the endpoints). 
Note that the underlined terms are the leading terms of the indicated quadratic binomials.  In terms of the circular embedding of $K_n$, these binomials correspond to replacing a noncrossing pair of edges with a crossing pair of edges, as illustrated.

\begin{figure}[h]
\includegraphics[width = 6cm]{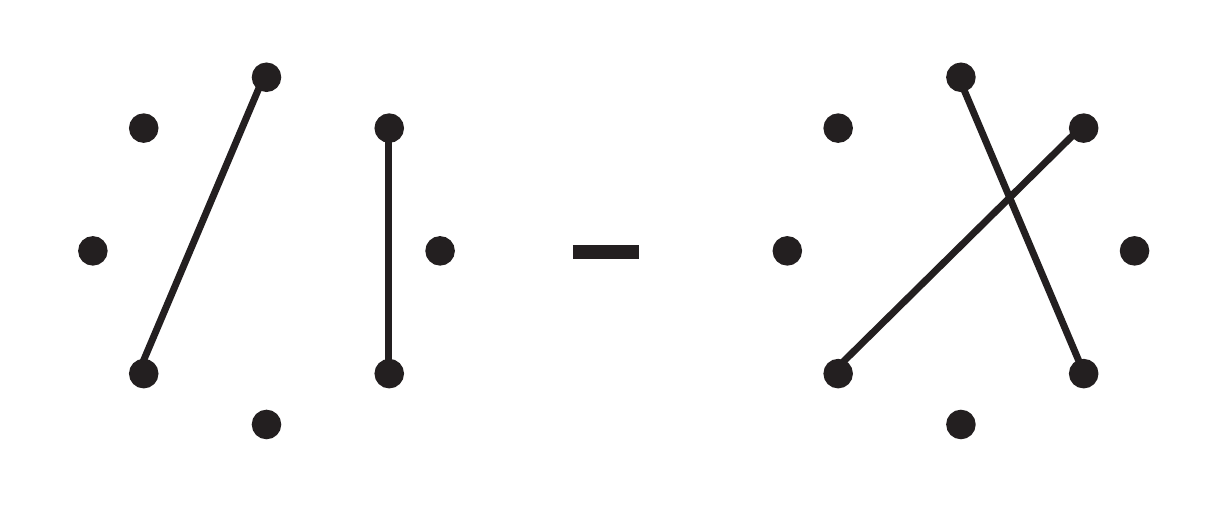}
\end{figure}

The notion of crossing leads to a simple description of ${\rm in}_\prec(I_n)$.

\begin{cor}
The monomial ideal ${\rm in}_\prec(I_n)$ is generated by all noncrossing pairs in the circular embedding of $K_n$;  that is,
$${\rm in}_\prec(I_n)  =  \left<  x_{ij} x_{kl}  \, \, | \, \, ij \mbox{ does not cross } kl  \right>.$$
\end{cor}

To prove our results about the Gr\"obner bases of the secant ideal $I_n^{\{2\}}$ and symbolic powers $I_n^{(2)}$, we want to employ the ``delightful'' strategy described in \cite{Sturmfels2006} and \cite{Sullivant2007}.  The idea here is to take advantage of the following proposition:

\begin{prop}\label{prop:delight}
\begin{enumerate}
\item \cite{Simis2000, Sturmfels2006} For any term order $\prec$ and any ideal $I$,  ${\rm in}_\prec(I^{\{r\}})  \subseteq {\rm in}_\prec(I)^{\{r\}}.$
\item  \cite{Sullivant2007}  Suppose that $I$ and ${\rm in}_\prec(I)$ are radical and $\kk$ is algebraically closed.  Then  ${\rm in}_\prec(I^{(r)})  \subseteq {\rm in}_\prec(I)^{(r)}.$
\end{enumerate}
\end{prop}

Using Proposition \ref{prop:delight}, if  we can find a collection of polynomials $\mathcal{G} \subset I^{\{r\}}$ such that $\left< {\rm in}_\prec(g) \,\, | \, \, g \in \mathcal{G}  \right>  =  {\rm in}_\prec(I)^{\{r\}}$, we can immediately deduce that $\mathcal{G}$ is a Gr\"obner basis for $I^{\{r\}}$ with respect to $\prec$,  (and similarly for the symbolic power.)  A term order $\prec$ is called $r$-delightful for $I$ when the equality ${\rm in}_\prec(I^{\{r\}})  = {\rm in}_\prec(I)^{\{r\}}$ holds.    Our goal will be to prove that the circular term order is $2$-delightful for $I_n$ by finding a combinatorial description for ${\rm in}_\prec(I_n)^{\{2\}}$ and ${\rm in}_\prec(I_n)^{(2)}$ and producing the polynomials whose initial terms generate these ideals.

The ideal ${\rm in}_\prec(I_n)$ is generated by squarefree quadrics so it is an example of an \emph{edge ideal}.  In general, an edge ideal is associated to an undirected graph $G$, as the ideal $I(G)  =  \left< x_i x_j  \,\, | \, \, ij \in E(G)  \right>$.  In the case of the initial ideal ${\rm in}_\prec(I_n)$, the corresponding graph is the \emph{non-crossing graph}.

\begin{defn}
The noncrossing graph $G_n$ has vertex set $V(G_n)  =  { [n] \choose 2}$ consisting of all two element subsets of $[n]$.  A pair $ij$ and $kl$ form an edge of the noncrossing graph if and only if the edges $ij$ and $kl$ do not cross in the circular embedding of the complete graph $K_n$. 
\end{defn}

The combinatorics of the secant ideals and symbolic powers of edge ideals are reasonably well-understood.  First, we deal with the case of the secant ideals.

\begin{thm}  \cite{Simis2000, Sturmfels2006}\label{thm:oddcycle}
Let $I(G)$ be an edge ideal.  Then 
$$I(G)^{\{2\}}  =  \left<  x_V  \,\, | \, \,  V \subseteq[n],  G_V \emph{ is an odd cycle } \right>.$$
\end{thm}

Here $G_V$ is the induced subgraph of $G$ with vertex set $V$ and $x_V = \prod_{i \in V} x_i$.  Thus, to describe the minimal generators of ${\rm in}_\prec(I_n)^{\{2\}}$, we must characterize the odd cycles in the noncrossing graph.

\begin{cor}\label{cor:secantinitial}
When $n = 5$, ${\rm in}_\prec(I_n)^{\{2\}} =  \left< x_{12}x_{23}x_{34}x_{45}x_{15} \right>$ has a single generator of degree five.  For $n \geq 6$ the  secant ideal ${\rm in}_\prec(I_n)^{\{2\}}$ has generators of every odd degree between $3$ and $n$, inclusive.  The generators come in two types:
\begin{enumerate}
\item Degree three generators: Let  $i_1 < i_2 < i_3 < i_4 < i_5 < i_6$.  Then $x_{i_1i_2} x_{i_3i_4} x_{i_5i_6}$ and $x_{i_1i_6} x_{i_2i_5} x_{i_3i_4}$ are degree three generators of ${\rm in}_\prec(I_n)^{\{2\}}$.
\item  Higher degree generators:  Let $i_1 \leq j_1 < i_2 \leq j_2 < \cdots <  i_{2k +1} \leq j_{2k+1}$, where $k >1 $.  Then  the monomial $\prod_{l = 1}^{2k+1}  x_{i_l j_{l +k-1}}$ is a minimal generator of ${\rm in}_\prec(I_n)^{\{2\}}$.  (The index $l + k -1$ is interpretted modulo $2k +1$.)  
\end{enumerate}
\end{cor}

Note that in the sequence
$i_1 \leq j_1 < i_2 \leq j_2 < \cdots <  i_{2k +1} \leq j_{2k+1} < i_1$ we assume that the sequence makes exactly one full revolution around the circle.  We call such a sequence of indices \emph{admissible} .

\begin{proof}
According to Theorem \ref{thm:oddcycle}, we must classify the induced odd cycles in the noncrossing graph $G_n$.  First, the cycles of length three:  this is a set of three edges in the circular embedding of $K_n$, none of which cross each other.  Since six distinct indices appear in a set of three disjoint edges, this reduces to the graph $K_6$.   In $K_6$ there are exactly two combinatorial types of non-crossing triples, the first isomorphic to the triple $12, 34, 56$ and the second isomorphic to the triple $16, 25, 34.$  These account for all the degree three monomials in ${\rm in}_\prec(I_n)^{\{2\}}$.

Now consider a cycle in the noncrossing graph $G_n$ of odd length $2k+1$ greater than or equal to five.  
We may suppose that each vertex appears in exactly one edge, since any vertex that appeared in two edges could have edges extended to two new vertices beyond to arrive at a cycle without that repeated vertex.  

Suppose that we have three consecutive edges in this cycle.  These edges will always form, up to a rotation and relabelling of vertices, a set of edges with exactly one crossing, like:
$k_1 k_3, k_2 k_4, k_5k_6$
where $k_1 < k_2< k_3 < k_4 < k_5 < k_6$ are arranged in circular order.  Since every other edge in the cycle must cross edge $k_5k_6$, this implies that there are $2k-2$ vertices on one side of $k_5k_6$ and $2k+2$ vertices on the other side.  By symmetry this pattern holds true for all edges in the graph.  In particular, the parity of the connections implies that we can alternately label the indices appearing by an $i$ or a $j$, splitting the vertices into two classes such that every edge is incident to both classes.  Furthermore, to guarantee the desired crossing conditions, after choosing the $i$ and $j$  for one of the edges so that edge has the form $i_1 j_k$, this forces the other edges to have the form $i_l j_{l +k -1}$ to guarantee each vertex is only contained in one edge.
\end{proof}

The small symbolic powers of edge ideals are also easy to characterize.

\begin{thm}\cite{Sullivant2007}
Let $G$ be a graph.  Then $I(G)^{(2)}  = I(G)^2 + I(G)^{\{2\}}$.  In particular, $I(G)^{(2)}$ is generated by degree three monomials $x_{i} x_{j} x_{k}$ where $i,j,k$ are a cycle in $G$, and degree 4 monomials $x_ix_j x_kx_l$ where $ij$ and $kl$ are (not necessarily  vertex disjoint) edges in $G$. 
\end{thm}

\begin{cor}\label{cor:symbinitial}
The symbolic square of the initial ideal  ${\rm in}_\prec(I_n)^{(2)}$ is generated by noncrossing triples and pairs of noncrossing edge pairs in the noncrossing graph $G_n$.
\end{cor}


\section{Master Polynomials}

In this section, we describe the master polynomials (Definition \ref{def:master}) , which are a collection of polynomials in the ideal $I^{\{2\}}_n$ whose initial terms generate the secant ideal  ${\rm in}_\prec(I_n)^{\{2\}}$.  This will allow us to complete the delightful strategy and prove:

\begin{thm}\label{thm:delight}
The circular term order $\prec$ is $2$-delightful for the ideal $I_n$ of the second hypersimplex.  In particular, the master polynomials and $3 \times 3$ off diagonal minors form a Gr\"obner basis for $I_n^{\{2\}}$ with respect to any circular term order.
\end{thm}

Let $i_1 \leq j_1 < i_2 \leq j_2 < \cdots <  i_{2k +1} \leq j_{2k+1}$ be an admissible circular sequence which admits a monomial in  ${\rm in}_\prec(I_n)^{\{2\}}$.  To begin with, we will assume that $i_1 < j_1,  i_2 < j_2, \ldots$.  Let $\bfi$ be the sequence of $i$ indices and $\bfj$ the sequence of $j$ indices.  We will construct a polynomial  $f_{\bfi,\bfj}$ in $I_{n}^{\{2\}}$ whose initial term is $\prod_{l = 1}^{2k+1}  x_{i_l j_{l +k-1}}$.

Since the elements of $\bfi$ and $\bfj$ are all distinct, we think about the monomial $\prod_{l = 1}^{2k+1}  x_{i_l j_{l +k-1}}$ as a fixed-point free involution $\prod_{l =1}^{2k+1} (i_l j_{l + k -1})$ in the symmetric group $S_{\bfi, \bfj}$ on the letters $\{i_1, j_1, i_2, j_2, \ldots, \ldots \}$.  In the symmetric group, consider the subgroup $Z_{\bfi, \bfj}$ generated by the $2k+1$ transpositions $(i_l j_{l-1})$.  Note that $Z_{\bfi, \bfj}  \cong \zz_2^{2k+1}$ because each of the $2k+1$ transpositions acts on a disjoint set of indices.

Let the group $Z_{\bfi, \bfj}$ act on $S_{\bfi, \bfj}$ by conjugation.  Since conjugation preserves cycle type, any conjugate of $\prod_{l =1}^{2k+1} (i_l j_{l + k -1})$ is a fixed-point free involution, and corresponds to a monomial in $\cc[x]$ of degree $2k+1$.  Let
$\bfx_{\bfi, \bfj}  = \prod_{l = 1}^{2k+1}  x_{i_l j_{l +k-1}}$
and for $\sigma \in S_{\bfi, \bfj}$ let $\bfx_{\bfi, \bfj}^\sigma$ be the monomial that is obtained from the fixed-point free involution $\sigma \cdot  \left( \prod_{l = 1}^{2k+1}  (i_l j_{l +k-1} ) \right) \cdot \sigma^{-1}$.

\begin{defn}\label{def:master}
Let $\bfi, \bfj$ be an admissible pair of sequences of length $2k+1$.  The master polynomial $f_{\bfi, \bfj}$ is the polynomial
$$f_{\bfi, \bfj}  =  \sum_{\sigma \in Z_{\bfi, \bfj} }   (-1)^\sigma  \bfx_{\bfi, \bfj}^\sigma.$$
\end{defn}

\begin{ex}
The first case where $f_{\bfi, \bfj}$ is nonzero occurs when $k =1$, with $\bfi = (1,3,5)$ and $\bfj = (2,4,6)$.  In this case, the master polynomial is an eight term cubic:
\begin{eqnarray*}
f_{\bfi,\bfj} & = & \, \, \, \, \, \underline{x_{12}x_{34}x_{56} }  \\
  &  &   - x_{12} x_{35} x_{46} -  x_{13}x_{24} x_{56} -  x_{15}x_{26} x_{34}  \\
  &  &    +  x_{13}x_{25}x_{46} + x_{14} x_{26} x_{35} + x_{15}x_{24}x_{36} \\
  &   &  -  x_{14} x_{25} x_{36}
\end{eqnarray*}

For $k =2$, the generic case, where all indices are distinct, yields a polynomial with $32$ terms.   
In the degenerate case where $\bfi = (1,2,3,4,5)$ and $\bfj = (1,2,3,4,5)$, there is extensive cancellation, and the resulting quintic has only twelve nonzero terms:
\begin{eqnarray*}
f_{\bfi, \bfj} & = &   \hspace{10pt}
 \underline{x_{12} x_{15} x_{23} x_{34} x_{45}}
- x_{12} x_{13} x_{25} x_{34} x_{45}
- x_{12} x_{14} x_{23} x_{35} x_{45}
+ x_{12} x_{14} x_{25} x_{34} x_{35} \\
 &  & 
+x_{12} x_{13} x_{24} x_{35} x_{45}
- x_{12} x_{15} x_{24} x_{34} x_{35}
+ x_{13} x_{14} x_{23} x_{25} x_{45}
- x_{13} x_{14} x_{24} x_{25} x_{35} \\
 &  & 
- x_{13} x_{15} x_{23} x_{24} x_{45}
+ x_{13} x_{15} x_{24} x_{25} x_{34}
- x_{14} x_{15} x_{23} x_{25} x_{34}
+ x_{14} x_{15} x_{23} x_{24} x_{35}.
\end{eqnarray*}
In the statistical literature on factor analysis, this degree five polynomial constraint on covariance matrices is known as a \emph{pentad}.  \qed
\end{ex}

To show that the master polynomials belong to the secant ideal $I_n^{\{2\}}$ we exploit the connection between secant ideals and prolongations.

\begin{thm}\cite[Thm 4.1]{Sidman2006}\label{thm:prolong}
Let $f$ be a homogeneous polynomial of degree $r(d-1) +1$ such that
$\frac{\partial^{\bfa} f}{\partial x^\bfa}  \in I$   for all   $\bfa \in \nn^n $ with $ \sum a_i  \leq (r-1)(d-1)$.  Then $f \in I^{\{r\}}$.
\end{thm}

\begin{lemma}\label{lem:belong}
The master polynomial $f_{\bfi, \bfj}$ is nonzero and belongs to the secant ideal $I_n^{\{2\}}$.
\end{lemma}

\begin{proof}
We will begin with the case that all the indices in $\bfi, \bfj$ are distinct, and then derive the general case as a consequence.

First of all, we will show that $f_{\bfi, \bfj}$ has $2^{2k+1}$ nonzero terms when the $\bfi, \bfj$ are distinct, and hence is non-zero.  This will be implied by the fact that the stabilizer of $\prod_{l = 1}^{2k+1}  (i_l j_{l +k-1} )$ with respect to the conjugation action of  $Z_{\bfi, \bfj}$ is trivial.  This can be seen by looking at the cyclic crossing numbers of the permutation $\prod_{l = 1}^{2k+1}  (i_l j_{l +k-1} )$ and its $Z_{\bfi, \bfj}$ conjugates.  To explain these cyclic crossing numbers, we cyclically arrange the numbers $i_1, j_1,  i_2, j_2, \ldots$, around a circle.  The cyclic crossing number of a fixed point free involution is the number of pairs of edges that cross in the associated embedded graph.  
The involution $\prod_{l = 1}^{2k+1}  (i_l j_{l +k-1} )$ has crossing number ${ 2k +1 \choose 2}  - ( 2k + 1)$ because it corresponds to a cycle in the noncrossing graph $G_n$.  Now  if $\sigma \in Z_{\bfi, \bfj}$ is the maximal permutation $\sigma =  \prod_{i =1}^{2k+1}  (i_l,j_{l-1})$  then
$$  \sigma \cdot \left(\prod_{l = 1}^{2k+1}  (i_l j_{l +k-1} )  \right) \cdot  \sigma^{-1}  =  \prod_{l = 1}^{2k+1}  (j_l i_{l +k-1} )$$
which has crossing number ${ 2k +1 \choose 2} $ since every pair of edges cross.  Now conjugation of any fixed-point free involution by an outside transposition $(i_l j_{l-1})$ changes the crossing number by at most one.  Thus, if $\sigma \in Z_{\bfi, \bfj}$ is the product of $m$ distinct $(i_l j_{l-1})$,  $  \sigma \cdot \left( \prod_{l = 1}^{2k+1}  (i_l j_{l +k-1} ) \right) \cdot  \sigma^{-1} $ must have crossing number ${2k +1 \choose 2} - (2k+1) + m$.  This implies that the stabilizer is trivial as claimed.

Now we will show that $f_{\bfi, \bfj} \in I_n^{\{2\}}$ using Theorem \ref{thm:prolong}.  Since $f_{\bfi, \bfj}$ has degree $2k+1$, we must show that all partial derivatives of $f_{\bfi, \bfj}$ up to order $k$ belong to $I_n$.  In the situation where $\bfi, \bfj$ are all distinct, every monomial in $f_{\bfi, \bfj}$ is squarefree, so we only need to consider the squarefree differential operators $D_\bfa = \frac{\partial^\bfa}{\partial \bfx^\bfa}$  where $\bfa$ is a $(0,1)$-vector.  This is given by a set of $\leq k$ edges in the complete graph $K_{\bfi, \bfj}$ with $4k+2$ vertices ${\bfi, \bfj}$.  Since there are only $\leq k$ edges, at least one of the edges $(i_l, j_{l-1})$ is not incident to any of these edges.  Now suppose that $D_\bfa f_{\bfi, \bfj}$ is not the zero polynomial.  Since $f_{\bfi, \bfj}$ is the signed sum of squarefree monomials, so also is $D_\bfa f_{\bfi, \bfj}$.  As $(i_l, j_{l-1})$ does not involve the variable set under differentiation by $D_\bfa$, $(i_l, j_{l-1})$ acts as an involution without fixed points on the set of monomials appearing in $D_\bfa f_{\bfi, \bfj}$, which yields a pairing of the monomials.  Since a monomial $m$ and its conjugate $m^{(i_l, j_{l-1})}$ will have oppositely signed coefficients,  this provides a decomposition of $D_\bfa f_{\bfi, \bfj}$ into a sum of binomials $D_\bfa f_{\bfi, \bfj} = \sum (\bfx^\bfu  - \bfx^\bfv)$ and each binomial belongs to the toric ideal $I_n$, because the set of indices appearing in $\bfx^\bfu$ and $\bfx^\bfv$ is the same.  This implies that 
$f_{\bfi, \bfj} \in I_n^{\{2\}}$.

Finally, we need to extend the result to the degenerate situation where for some values of $l$, $i_l = j_l$.  First of all, the fact that $f_{\bfi, \bfj} \in I_n^{\{2\}}$ follow just by ``identifying parameters''.  If we think about the $\bfi$ and $\bfj$ as being generic indices, $f_{\bfi, \bfj} \in I_n^{\{2\}}$ if and only if $f_{\bfi, \bfj}( x_{ij}   =t_i t_j + u_i u_j)  = 0$ for all vectors $t$ and $u$.  In particular, this holds if some of the $t_i = t_ j$ and $u_i = u_j$, which is what happens when we move from the generic indices to degenerate indices.  To show that, $f_{\bfi, \bfj}$ is nonzero in the degenerate case, we go back to counting cyclic crossings.  Note that setting some $i_l = j_l$ for some values of $l$ cannot decrease the cyclic crossing number of the resulting graph and the crossing number might increase.  However, for the graph corresponding to the monomial  $\prod_{l = 1}^{2k+1}  x_{i_l j_{l +k-1}}$  the cyclic crossing number remains the same, even if some $i_l = j_l$.  This implies that this particular monomial does not cancel with any other monomial in the representation, so $f_{\bfi, \bfj}$ is not zero.
\end{proof}

\begin{lemma}\label{lem:init}
The leading term of the master polynomial $f_{\bfi, \bfj}$ with respect to any circular term order is the cycle monomial $\bfx_{\bfi, \bfj}  =  \prod_{l = 1}^{2k+1}  x_{i_l j_{l +k-1}}$.
\end{lemma}

\begin{proof}
The monomial $\bfx_{\bfi, \bfj}$ is the only monomial appearing in $f_{\bfi, \bfj}$ that is divisible by one of the odd cycle generators of ${\rm in}_\prec(I_n)^{\{2\}}$.  Indeed, applying a nonzero element $\sigma \in Z_{\bfi, \bfj}$ removes edges from the noncrossing graph of this involution.   Since $f_{\bfi, \bfj} \in I_n^{\{2\}}$, and we must have the containments ${\rm in}_\prec(I_n^{\{2\}}) \subseteq {\rm in}_\prec(I_n)^{\{2\}}$, this implies $\bfx_{\bfi, \bfj}$ is the initial term.

Another argument is that every monomial appearing in $f_{\bfi,\bfj}$ can be obtained from $\bfx_{\bfi, \bfj}$ by applying some sequence of quadratic reductions from the Gr\"obner basis for the second hypersimplex.  Since such any such single reduction takes a monomial and produces a monomial that is smaller in the term order, this implies that $\bfx_{\bfi, \bfj}$ is the leading monomial.
\end{proof}

\noindent {\em Proof of Theorem  \ref{thm:delight}.}  We employ the delightful strategy to determine the Gr\"obner basis for $I_n^{\{2\}}$.  In particular, to show that the circular term order $\prec$ is $2$-delightful, we must show that each of the monomials in ${\rm in}_\prec(I_n)^{\{2\}}$, that were determined in Corollary  \ref{cor:secantinitial} is the initial term of some polynomial in $I_n^{\{2\}}$.  

By Lemmas \ref{lem:belong} and \ref{lem:init}, each admissible sequence of odd length $5$ or greater produces a master polynomial $f_{\bfi,\bfj}$ whose initial term is the corresponding odd cycle in the noncrossing graph $G_n$.  Also, admissible sequences of length $3$ produce master polynomials of degree $3$ whose initial terms of the monomials of form $x_{i_1i_2} x_{i_3i_4} x_{i_5i_6}$.  So to finish the proof, we need only show that each degree $3$ monomial of the form  $x_{i_1i_6} x_{i_2i_5} x_{i_3i_4}$ is the initial monomial of some $3 \times 3$ off diagonal minor.  So, consider the off diagonal minor
$$X =  \begin{pmatrix}
x_{i_1i_4} & x_{i_1i_5} & x_{i_1i_6} \\
x_{i_2i_4} & x_{i_2i_5} & x_{i_2i_6} \\
x_{i_3i_4} & x_{i_3i_5} & x_{i_3i_6}
\end{pmatrix}$$
whose antidiagonal term is the desired monomial.  This off-diagonal minor clearly belongs to $I_n^{\{2\}}$ since it vanishes on rank two symmetric matrices.  Furthermore, the term $x_{i_1i_6}x_{i_2i_5} x_{i_3i_4}$ is its initial term, since each of the five other terms can be obtained from it by applying a sequence of non-crossing to crossing moves.  These non-crossing to crossing moves send a monomial to something smaller in the term order, because they amount to Gr\"obner reduction with respect to the circular Gr\"obner basis for $I_n$.  \qed

\smallskip

Finally, we are able to apply the results about the secant ideal $I_n^{\{2\}}$ and its delightful Gr\"obner basis to also deduce the Gr\"obner basis for the symbolic power $I_n^{(2)}$.

\begin{cor}
The set of $3 \times 3$ off-diagonal minors, the degree $3$ master polynomials, and the products of pairs of $2 \times 2$ off-diagonal minors form a Gr\"obner basis for $I_n^{(2)}$ with respect to any circular term order.  Furthermore  $I_n^{(2)}  =  I_n^{2} + I_n^{\{2\}}$.
\end{cor}

\begin{proof}
By Corollary \ref{cor:symbinitial}, the symbolic power of the initial ideal is ${\rm in}_\prec(I_n)^{(2)} =
{\rm in}_\prec(I_n)^{2}  + {\rm in}_\prec(I_n)^{\{2\}}$, which is generated by length three cycles in the noncrossing graph plus products of pairs of noncrossing edges.  Each length three cycle is the initial term of either a degree three master polynomial or a $3 \times 3$ off-diagonal minor.  Products of pairs of noncrossing edges are the initial terms of products of $2 \times 2$ minors.  All these polynomials belong to $I_n^{(2)}$ by the containment $I^2 + I^{\{2\}} \subseteq I^{(2)}$ (see, for example, \cite{Sullivant2007}).
This implies that ${\rm in}_\prec(I_n)^{(2)} = {\rm in}_\prec(I_n^{(2)}) $ and hence that the desired polynomials form a Gr\"obner basis for $I_n^{(2)}$.  
\end{proof}



\begin{thebibliography}{99}


\bibitem{Alexander1995}
J.~Alexander and A.~Hirschowitz. Polynomial interpolation in several variables.  \emph{J.
Algebraic Geom.} {\bf 4} (1995), no. 2, 201Ð222.

\bibitem{Catalisano2002}  M.~V.~Catalisano, A.~V.~Geramita, and A.~Gimigliano. Ranks of tensors, secant varieties
of Segre varieties and fat points. \emph{Journal of Linear Algebra and Its Applications}
{\bf 355} (2002), 263Ð285.

\bibitem{Deloera1995}  J.~De Loera, B.~Sturmfels, and R.~Thomas.  Gr\"obner bases and triangulations of the second hypersimplex.    \emph{Combinatorica}  {\bf 15}  (1995),  no. 3, 409--424. 

\bibitem{Drton2007}  M.~Drton, B.~Sturmfels, and S.~Sullivant.  Algebraic factor analysis: tetrads, pentads and beyond.  \emph{Probab. Theory Related Fields} {\bf 138} (2007), no. 3-4, 463--493. 

\bibitem{Landsberg2003}  J.~M.~Landsberg and L.~Manivel.  On the projective geometry of rational homogeneous varieties. \emph{Comment. Math. Helv.} {\bf 78} (2003), no. 1, 65--100

\bibitem{Landsberg2004}
J.~M.~Landsberg and L.~Manivel.  On the ideals of secant varieties to Segre varieties.
\emph{Foundations of Computational Mathematics.} {\bf 4} (2004), no. 4, 397Ð422.



\bibitem{Sidman2006}  J.~Sidman and S.~Sullivant.   Prolongations and computational algebra.  To appear in \emph{Canadian Mathematical Journal}, 2006.




\bibitem{Simis2000}  A.~Simis and B.~Ulrich.  On the ideal of an embedded join.  \emph{J. Algebra} {\bf 226} (2000)  1--14.



\bibitem{Sturmfels2006}  B.~Sturmfels and S.~Sullivant.  Combinatorial secant varieties. {\em Quarterly Journal of Pure and Applied Mathematics}  {\bf  2}  (2006) 285--309, (Special issue: In Honor of Robert MacPherson).

\bibitem{Sullivant2007}  S.~Sullivant.  Combinatorial symbolic powers.   {\em J. Algebra}  {\bf 319}  (2008),   115--142.


\end{thebibliography}
\end{document}